\documentclass[12pt,twoside,lineno]{amsart}
\usepackage{amsmath,amssymb}
\usepackage{mathrsfs}
\usepackage{caption}
\usepackage{amsmath}
\usepackage{amsthm}
\usepackage{epsfig}
\usepackage{graphicx}
\newtheorem{thm}{Theorem}[section]
\newtheorem{lem}[thm]{Lemma}
\newtheorem{cor}[thm]{Corollary}

\newtheorem*{remark}{Remark}

\newtheorem{ex}{Example}

\title[The Euler-Glaisher Theorem over Totally Real Number Fields]%
{The Euler-Glaisher Theorem over Totally Real Number Fields}
\author[Se Wook Jang]{Se Wook Jang}\address{Department of Mathematics, Gangneung-Wonju National University, Gangneung 210-702, Korea}\email{dcgg@naver.com}
\author[Byeong Moon Kim]{Byeong Moon Kim}\address{Department of Mathematics, Gangneung-Wonju National University, Gangneung 210-702, Korea}\email{kbm@gwnu.ac.kr} 
\author[Kwang Hoon Kim]{Kwang Hoon Kim}\address{Department of Mathematics, Gangneung-Wonju National University, Gangneung 210-702, Korea}\email{h2078@naver.com}

\thanks{This work was supported by the National Research Foundation of Korea(NRF) grant funded by the Korea government(MSIT) (No. RS-2023-00247457)
}
\begin{document}
\maketitle

\begin{abstract}
In this paper, we study the partition theory over totally real number fields. Let $K$ be a totally real number field. A partition of a totally positive algebraic integer $\delta$ over $K$ is $\lambda=(\lambda_1,\lambda_2,\ldots,\lambda_r)$  for some totally positive integers $\lambda_i$ such that $\delta=\lambda_1+\lambda_2+\cdots+\lambda_r$. We find an identity to explain the number of partitions of $\delta$ whose parts do not belong to a given ideal $\mathfrak a$. We obtain a  generalization of the Euler-Glaisher Theorem over totally real number fields as a corollary. We also prove that the number of solutions to the equation $\delta=x_1+2x_2+\cdots+nx_n$ with $x_i$ totally positive or $0$ is equal to that of chain partitions of $\delta$. A chain partition of $\delta$ is a partition $\lambda=(\lambda_1,\lambda_2,\ldots,\lambda_r)$ of $\delta$ such that $\lambda_{i+1}-\lambda_i$ is totally positive or $0$. 
\end{abstract}

\section{introduction}
The theory of partitions of a natural number is an interesting area in additive number theory with many beautiful identities, theorems and  conjectures. It is initiated by Leibniz and the first important results are those of Euler. It is famous that Ramanujan made this area more plentiful by suggesting some astounding conjectures which disclose the magic of natural numbers.

A partition $\lambda$ of a positive rational integer $n$ is a $k$-tuple $(\lambda_1,\lambda_2,\cdots,\lambda_k)$ of  positive integers $\lambda_i$($i=1,2,\cdots,k$) such that $n=\lambda_1+\lambda_2+\cdots+\lambda_k$. For each $i$, $\lambda_i$ is a part of  $\lambda=(\lambda_1,\lambda_2,\cdots,\lambda_k)$. The number $k$ of the parts of a partition $\lambda$  is not specified in general. Two partitions 
$\lambda=(\lambda_1,\lambda_2,\cdots,\lambda_k)$ and $\lambda'=(\lambda'_1,\lambda'_2,\cdots,\lambda'_{k'})$ are the same if $k=k'$ and $\lambda'$ is obtained by some proper reordering of  $\lambda$. For a nonempty subset $H$ of $\mathbb Z^+$, the partition number  $p(''H'',n)$ is the number of all distinct partitions of $n$ whose parts belong to $H$. We write $p(''{\mathbb Z^+}'',n)$ as $p(n)$. There are various partition numbers defined by the numbers of all distinct partitions satisfying certain restrictions. It is an interesting area of partition theory to study the relations between these diverse partition numbers. Interpreting the coefficients of $q$-series identities, especially defined by infinite products by these miscellaneous partition numbers is a fundamental problem in partition theory. The identity 
\[\prod_{n=1}^{\infty}\frac1{1-q^n}=\sum_{n=0}^{\infty}p(n)q^n\] 
is a starting point of partition theory. Although the partition theory has been developed in various ways, most of the current results are acquired for positive rational integers. In 1950 Rademacher \cite{rad} developed the partition theories over real quadratic fields. Ever since, some results on the partitions over real quadratic fields \cite{mit} have been added. This paper is an attempt to extend some partition theory, especially Euler-Glaisher Theorem to general totally real number fields. Let $K$ be a totally real number field and $\mathcal O^+=\mathcal O^+_K$ be the set of all totally positive algebraic integers over $K$. For each nonempty subset $H$ of $\mathcal O^+$ and $\delta\in\mathcal O^+$, the partition $\lambda=(\lambda_1,\lambda_2,\ldots,\lambda_k)$ of $\delta$ in $H$ and the partition number $p(''H'',\delta)$ are defined in the same way as in the case of natural numbers. We define formal $q-$sums, a generalization of $q-$series, in $H$. We also define the generating function $\sum_{\delta\in\mathcal O^+}p(''H'',\delta)q^{\delta}$ of partitions in $H$ by showing the finiteness of partitions of $\delta$ in $H$ and by proving some basic properties of formal $q-$sums. We prove that for each ideal $\mathfrak a$ of $\mathcal O$ and $d\in\mathbb Z^+$ such that $d\in\mathfrak a$, $p(''\mathcal O^+\setminus\mathfrak a'',\delta)$ is equal to the number $p(''S''(\le d-1),\delta)$ of partitions of $\delta$ in $S$ such that the same parts are admitted at most $d-1$ times for some subset $S$ of $\mathcal O^+$, which is given explicitly. If $\mathfrak a$ is an ideal $(d)$ of $\mathcal O$ generated by $d$, then $S=\mathcal O^+$, and we obtain the natural generalization of the Euler-Glaisher Theorem over $K$. We consider another kind of partitions of $\delta\in\mathcal O^+$, which we call chain partitions. A partition $\lambda=(\lambda_1,\lambda_2,\cdots,\lambda_k)$ of $\delta$ is a chain partition if $\lambda_{i+1}-\lambda_i$ is totally positive or $0$ for all $i=1,2,\cdots,k-1$. We prove that the number of distinct chain partitions of $\delta$ is the same as that of the solutions to $\delta=x_1+2x_2+\cdots+kx_k$ for $x_i\in\mathcal O^+\cup\{0\}$. 

\section{basic definitions and theorems}
 
A $\textit{partition}$ $\lambda$ of a totally positive algebraic integer $\delta$ over $K$ is a finite  sequence $\lambda_1, \lambda_2, \ldots, \lambda_k$ of totally positive integers over $K$ such that $\sum_{i=1}^k\lambda_i=\delta$. The $\lambda_i$ are called the $parts$ of $\lambda$. The partition function $p(\delta)$ is the number of partitions of $\delta$ up to changing the order. The partition $\lambda$ is denoted by $(\lambda_1,\lambda_2,\ldots,\lambda_k)$. 

Let $S$ be a set of partitions. We denote $p(S,\delta)$ by the number of partitions of $\delta$ that belong to $S$. Let $H$ be a set of totally positive integers over $K$. Let $''H''$ be the set of all partitions whose parts lie in $H$, and $''H''(\le d)$ be the set of all partitions in which no part appears more than $d$ times and each part is in $H$. For a guidance of the partition theory of positive rational integers, see \cite{and11}. This paper use the analogous notations introduced in the book. 

The following two theorems are classical theorems in partition theory of natural numbers. Theorem \ref{gl} is a generalization of Theorem \ref{eu}. The contents of this and next two sections are generalizations of these theorems to the set $\mathcal O^+$ of all totally positive algebraic integers over $K$.

\begin{thm}[Euler]\label{eu}
Let $\mathscr O$ be the set of all partitions with odd parts and $\mathscr D$ be the set of all partitions with distinct parts. Then $p(\mathscr O,n)=p(\mathscr D,n)$ for all $n\in\mathbb Z^+$.
\end{thm}

\begin{thm}[Glaisher]\label{gl}
Let $\mathbb Z^+_d$ denote the set of all positive integers not divisible by $d$. Then $p(''{\mathbb Z^+_{d+1}}'',n)=p(''{\mathbb Z^+}''(\le d),n)$ for all $n\in\mathbb Z^+$.
\end{thm}

Let $K$ be a finite extention field of $\mathbb Q$. Let $H$ be a subset of $\mathbb R$. To define the partition number $p(''H'',\delta)$ of $\delta$ in $H$, it is necessary that there are finitely many $(\alpha,\beta)\in H\times H$ such that $\alpha+\beta=\delta$. It is known that for a totally real number field $K$, the set $\mathcal O_K^+\cup\{0\}$ of all totally positive algebraic integers with $0$ satisfies there are finitely many $\alpha,\beta\in\mathcal O_K^+\cup\{0\}$ such that  $\alpha+\beta=\delta$ for all $\delta\in\mathcal O_K^+\cup\{0\}$. The following lemma and corollary are given in \cite{kim}. But, since it is not published yet, we prove it anew.

\begin{lem}\label{ndissol}
For $n,a_1,a_2,\ldots,a_n\in\mathbb Z^+$, there are finitely many positive integers $a_0$ such that the equation $a_nx^n-a_{n-1}x^{n-1}+a_{n-2}x^{n-2}-\cdots+(-1)^na_0=0$ has $n$ distinct positive solutions.
\end{lem}
\begin{proof}
Let $f(x)=a_nx^n-a_{n-1}x^{n-1}+a_{n-2}x^{n-2}-\cdots+(-1)^na_0$. If $f(x)=0$ has $n$ distinct positive solutions, then $f'(x)=0$ has $n-1$ distinct positive solutions. Let $\alpha$ be the smallest solution of $f'(x)=0$. Then $f(0)<0<f(\alpha)$ if $n$ is odd, and $f(0)>0>f(\alpha)$ if $n$ is even. Since $a_0=(-1)^nf(0)>0$ and \begin{align*}
0&<(-1)^{n-1}f(\alpha)=(-1)^{n-1}f(\alpha)+(-1)^nf(0)-a_0\\
&=(-1)^{n-1}(f(\alpha)-f(0))-a_0,
\end{align*}
we have $0<a_0<(-1)^{n-1}(a_n\alpha^n-a_{n-1}\alpha^{n-1}+a_{n-2}\alpha^{n-2}-\cdots+(-1)^{n-1}a_1\alpha)$. Thus, the number of the desired $a_0$ is finite.
\end{proof}

\begin{cor}
For each $n, M>0$, there are finitely many totally positive algebraic integers $\alpha$ such that $\deg(\alpha)=n$ and ${\rm{tr}}(\alpha)<M$.
\end{cor}
\begin{proof}
Let $\alpha$ be a  totally positive algebraic integer of degree $n$. Then there exists a polynomial $f(x)=x^n-a_{n-1}x^{n-1}+a_{n-2}x^{n-2}-\cdots+(-1)^na_0$ such that $f(\alpha)=0$ and $a_i\in\mathbb Z^+$. We can see $f^{(r)}(x)=\frac{n!}{(n-r)!}x^{n-r}-\frac{(n-1)!}{(n-r-1)!}a_{n-1}x^{n-r-1}+\cdots+(-1)^{n-r}r!a_r=0$ has $n-r$ distinct positive solutions. The number of positive integers $a_{n-1}$ satisfying ${\rm tr}(\alpha)=a_{n-1}<M$ is finite. The finiteness of $a_r$ follows from that of $a_{n-1}, a_{n-2}, \ldots,$ $a_{r+1}$ by Lemma \ref{ndissol}. By induction, we have that the number of the desired $\alpha$ is finite. 
\end{proof}

\begin{cor}\label{tr<M}
Let $K$ be a totally real number field. Then for all $M>0$, there are finitely many totally positive algebraic integers over $K$ whose trace is at most $M$.
\end{cor}

\begin{cor}\label{d=a+b}
For all $\delta\in\mathcal O^+$, there are finitely many $\alpha, \beta\in\mathcal O^+$ such that $\delta=\alpha+\beta$.
\end{cor}

\begin{cor}\label{finiteness}
For all $\delta\in\mathcal O^+$, the number $p(\delta)$ of partitions of $\delta$ in $\mathcal O^+$ is finite. Equivalently, for all $\delta\in\mathcal O^+$, there are finitely many finite sequences $\delta_1, \delta_2,\ldots,\delta_n$ in $\mathcal O^+$ such that $\delta=\delta_1+\delta_2+\cdots+\delta_n$.
\end{cor}

\section{formal power $q$-sums and Infinite product generating functions of one variable over totally real number fields}
A large part of the partition theory has relied on the use of $q$-series. The following theorem is a classic result.

\begin{thm}[\cite{and11},Theorem 1.1]\label{''H''}
Let $H$ be a set of positive rational integers, and let
\begin{align*}
f(q)&=\sum_{n\ge0}p(''H'',n)q^n,\\
f_d(q)&=\sum_{n\ge0}p(''H''(\le d),n)q^n.
\end{align*}
Then for $|q|<1$,
\begin{align*}
f(q)&=\prod_{n\in H}\frac1{1-q^n},\\
f_d(q)&=\prod_{n\in H}(1+q^n+\cdots+q^{dn})\nonumber\\
&=\prod_{n\in H}\frac{1-q^{(d+1)n}}{1-q^n}.
\end{align*}
\end{thm}

To generalize the previous theorem to totally real number fields, we define the $\textit{ring of}$ $\textit{formal}$ $\textit{power sums}$, or $\textit{the ring of formal q-sums}$. To study the partition theory of totally positive integers over totally real number field $K$, we modify the ring $\mathbb R[[q]]=\{a_0+a_1q+a_2q^2+\cdots|a_i\in\mathbb R\}$ of $q$-series as follows. Let $H\subset\mathcal O^+\cup\{0\}$. Assume that $0\in H$ and $H$ is closed under addition. We define the ring of formal power sums $\mathbb R[[q]]_H$ by
\begin{align*}
\mathbb R[[q]]_H=\left\{\sum_{\delta\in H}c_{\delta}q^{\delta}|c_{\delta}\in\mathbb R\right\}.
\end{align*}
The convergence of a formal power sum does not care because most of formal power sums does not converge except $K=\mathbb Q$. Let $f(q)=\sum_{\delta\in H}c_{\delta}q^{\delta}\in\mathbb R[[q]]_H$ and let $c(f(q),\delta)=c_{\delta}$ be the coefficient of $q^{\delta}$ in $f(q)$. The sum and product of two formal power $q$-sums are defined by
\begin{align*}
&\sum_{\delta\in H}c_{\delta}q^{\delta}+\sum_{\delta\in H}d_{\delta}q^{\delta}=\sum_{\delta\in H}(c_{\delta}+d_{\delta})q^{\delta},\\
&\left(\sum_{\delta\in H}c_{\delta}q^{\delta}\right)\left(\sum_{\delta\in H}d_{\delta}q^{\delta}\right)=\sum_{\gamma\in H}\left(\sum_{\gamma=\alpha+\beta}c_{\alpha}d_{\beta}\right)q^{\gamma}.
\end{align*}
Since $H\subset\mathcal O^+\cup\{0\}$ by Corollary \ref{d=a+b}, the product $f(q)g(q)$ is well-defined. If $\sum_{i\in I}a_{i,\delta}$ converges for all $\delta\in H$, then the infinite sum is defined by
\begin{align*}
&\sum_{i\in I}\left(\sum_{\delta\in H}a_{i,\delta}q^{\delta}\right)=\sum_{\delta\in H}\left(\sum_{i\in I}a_{i,\delta}\right)q^{\delta}.
\end{align*}
To define the infinite product $\prod_{i\in I}f_i(q)$ of $f_i(q)=\sum_{\delta\in H}a_{i,\delta}q^{\delta}$, we only need to define all the coefficeints $c_{\delta}$ of $q^{\delta}$ in 
\begin{align*}
\prod_{i\in I}f_i(q)=\sum_{\delta\in H}c_{\delta}q^{\delta}.
\end{align*}
If for each $\delta\in\mathcal O^+\cup\{0\}$ there is a finite set $I_{\delta}$ of $I$ such that the coefficeint of $q^{\delta}$ in $\prod_{i\in J}f_i(q)$ is the same to that of $\prod_{i\in I_{\delta}}f_i(q)$ for all $J$ with $i_{\delta}\subset J\subset I$, then each  $c_{\delta}$ is well-defined, and thus we can define $\prod_{i\in I}f_i(q)$. If there are finitely many $i\in I$ such taht $a_{i,0}\ne1$ and for each $\delta\in\mathcal O^+$ there are finitely many $i\in I$ such that $a_{i,\delta}\ne0$, then the above condition is satisfied. In this case, we can see that $c_{\delta}=\sum a_{i_1,\delta_1}a_{i_2,\delta_2}\cdots a_{i_n,\delta_n}$ is the sum of all the product $a_{i_1,\delta_1}a_{i_2,\delta_2}\cdots a_{i_n,\delta_n}$ such that $i_1, i_2,\ldots, i_n$ are distinct and $\delta=\delta_1+\delta_2+\cdots+\delta_n$. Suppose $f_i(q)=\sum_{\delta\in H}a_{i,\delta}q^{\delta}$ and $g_i(q)=\sum_{\delta\in H}b_{i,\delta}q^{\delta}$ satisfies the above condition. Then the product of two infinte products $\prod_{i\in I}f_i(q)$ and $\prod_{i\in I}g_i(q)$ is 
\begin{align*}
\left(\prod_{i\in I}f_i(q)\right)\left(\prod_{i\in I}g_i(q)\right)=\sum_{\delta\in H}c_{\delta}q^{\delta}
\end{align*}
where $c_{\delta}=\sum a_{i_1,\zeta_1}a_{i_2,\zeta_2}\cdots a_{i_k,\zeta_k}b_{j_1,\eta_1}b_{i_j,\eta_2}\cdots b_{j_l,\eta_l}$ is the sum over all $i_1,i_2,\ldots,i_k,j_1,j_2,\ldots, j_l\in I$ and nonzero $\zeta_1,\zeta_2,\ldots,\zeta_k, \eta_1,\eta_2,\ldots,\eta_l\in H$ such that $i_1,i_2,\ldots, i_k$ are distinct, $j_1,j_2,\ldots, j_l$ are distinct, $\delta=\zeta_1+\zeta_2+\cdots+\zeta_k+\eta_1+\eta_2+\cdots+\eta_l$, $a_{i,0}=1$ for $i\ne i_t$ and $b_{j,0}=1$ for $j\ne j_s$. We can see that $c_{\delta}$ is equal to the coefficient $d_{\delta}$ of $q^{\delta}$ in 
\begin{align*}
\prod_{i\in I}f_i(q)g_i(q)=\prod_{i\in I}\left(\sum_{\gamma=\alpha+\beta}a_{i,\alpha}b_{i,\beta}\right)q^{\gamma}=\sum_{\delta\in H}d_{\delta}q^{\delta}.
\end{align*}
Thus we have 
\begin{align*}
\prod_{i\in I}f_i(q)\prod_{i\in I}g_i(q)=\prod_{i\in I}f_i(q)g_i(q).
\end{align*}
In particular 
\begin{align*}
\prod_{\delta\in H}(1-q^{\delta})\prod_{\delta\in H}(1+q^{\delta}+q^{2\delta}+\cdots)=\prod_{\delta\in H}1=1.
\end{align*}
As a result we can use a convenient notation
\begin{align*}
\prod_{\delta\in H}\frac1{1-q^{\delta}}=\prod_{\delta\in H}(1+q^{\delta}+q^{2\delta}+\cdots)\in\mathbb R[[q]]_H.
\end{align*}
If $H$ is not closed under addition, then $\mathbb R[[q]]_H=\{\sum_{\delta\in H}a_{\delta}q^{\delta}|a_{\delta}\in\mathbb R\}$ is an abelian group under addition and not a ring. Products and infinte products of the elements of  $\mathbb R[[q]]_H$ belong to  $\mathbb R[[q]]_{\mathcal O^+\cup\{0\}}$.

The study of formal $q$-sums is closely related to the theory of partitions over $K$. In fact, we can think of a formal $q$-sum as a $q$-series over $K$.

The following lemma is a generalization of Theorem \ref{''H''} to totally real number fields. 
\begin{lem}\label{hd}
Let $H$ be a set of totally positive integers over $K$, and let
\begin{align*}
f(q)&=\sum_{\delta\in\mathcal{O}^+\cup\{0\}}p(''H'',\delta)q^{\delta}
\end{align*}
and
\begin{align*}
f_d(q)&=\sum_{\delta\in\mathcal{O}^+\cup\{0\}}p(''H''(\le d),\delta)q^{\delta}.
\end{align*}
Then 
\begin{align*}
f(q)&=\prod_{\delta\in H}\frac1{1-q^{\delta}}
\end{align*}
and
\begin{align*}
f_d(q)&=\prod_{\delta\in H}(1+q^{\delta}+\cdots+q^{d\delta})\\
&=\prod_{\delta\in H}\frac{1-q^{(d+1)\delta}}{1-q^{\delta}}.
\end{align*}
\end{lem}
The proof of Lemma \ref{hd} is the same as that of Theorem \ref{''H''} execpt that formal $q$-sums are used instead of $q$-series.

\section{ideal partitions}
The Glaisher Theorem can be differently stated that for $H=\mathbb Z^+\setminus(d)$, $p(''H'',n)=p(''{\mathbb Z^+}''(\le d-1),n)$ for all $n\in\mathbb Z^+$. Let $K$ be a totally real number field. We are interested in $p(''H'',\delta)$ for $H=\mathcal O^+\setminus\mathfrak a$ where $\mathfrak a$ is an ideal of $\mathcal O=\mathcal O_K$ as a generalization of Glaisher Theorem. It is equivalent to studing the generating function \[\prod_{\delta\in\mathcal O^+, \delta\notin\mathfrak a}\frac1{1-q^{\delta}}.\] In each infinite product given in this section, we assume that $\delta\in\mathcal O^+$. For example, $\delta\notin\mathfrak a$ means $\delta\in\mathcal O^+$ and $\delta\notin\mathfrak a$, and $\delta\in\mathfrak a$ means $\delta\in\mathcal O^+\cap\mathfrak a$. Thus the product above is written as \[\prod_{ \delta\notin\mathfrak a}\frac1{1-q^{\delta}}.\] 

Let $\mathfrak a$ be an ideal of $\mathcal O$. Choose an integer $d\in\mathbb Z^+\cap\mathfrak a$, then $(d)=\mathfrak p_1^{e_1}\mathfrak p_2^{e_2}\cdots\mathfrak p_k^{e_k}$ for some distinct prime ideals $\mathfrak p_1,\mathfrak p_2,\ldots,\mathfrak p_k$ of $\mathcal O$. We can see $\mathfrak a=\mathfrak p_1^{e_1'}\mathfrak p_2^{e_2'}\cdots\mathfrak p_k^{e_k'}$ such that $0\le e_i'\le e_i$ for all $i$. Let $\mathfrak p$ be a prime ideal of $\mathcal O$. For $\delta\in\mathcal O^*$, the $\mathfrak p$$\textit{-adic order}$ $\nu_{\mathfrak p}(\delta)$ of $\delta$ is defined by $\nu_{\mathfrak p}(\delta)=k$ if and only if $\delta\in\mathfrak p^k$ and $\delta\notin\mathfrak p^{k+1}$.

\begin{thm}
Let $\mathfrak a, d,\mathfrak p_i,e_i,e_i'$ be the same as above. Let $S=\mathcal O^+\setminus\cup_{j=0}^{\infty}\left(\mathfrak a(d^j)\setminus (d^{j+1})\right)$. Then 
\begin{align*}
\prod_{\delta\notin\mathfrak a}\frac1{1-q^{\delta}}=\prod_{\delta\in S}(1+q^{\delta}+q^{2\delta}+\cdots+q^{(d-1)\delta}).
\end{align*}
Thus, we have $p(''\mathcal O^+\setminus\mathfrak a'',\delta)=p(''S''(\le d-1),\delta)$ for all $\delta\in\mathcal O^+$.
\end{thm}
\begin{proof}
Let $S_n=\mathfrak a(d^{n-1})\setminus(d^n)$ for $n\in\mathbb Z^+$. Then $S=\mathcal O^+\setminus\cup_{n=0}^{\infty}S_n$.
By Lemma \ref{hd}, we have
\begin{align*}
\prod_{\delta\notin\mathfrak a}\frac1{1-q^{\delta}}&=\prod_{\delta\notin(d)}\frac1{1-q^{\delta}}\prod_{\delta\in\mathfrak a\setminus(d)}(1-q^{\delta})\\
&=\left(\prod_{\delta\in\mathcal O^+}\frac{1-q^{d\delta}}{1-q^{\delta}}\right)\prod_{\delta\in\mathfrak a\setminus(d)}\left(\frac{1-q^{\delta}}{1-q^{d\delta}}\frac{1-q^{d\delta}}{1-q^{d^2\delta}}\frac{1-q^{d^2\delta}}{1-q^{d^3\delta}}\cdots\right)\\
&=\left(\prod_{\delta\in\mathcal O^+}\frac{1-q^{d\delta}}{1-q^{\delta}}\right)\left(\prod_{\delta\in\mathfrak a\setminus(d)}\frac{1-q^{\delta}}{1-q^{d\delta}}\right)\left(\prod_{\delta\in\mathfrak a\setminus(d)}\frac{1-q^{d\delta}}{1-q^{d^2\delta}}\right)\left(\prod_{\delta\in\mathfrak a\setminus(d)}\frac{1-q^{d^2\delta}}{1-q^{d^3\delta}}\right)\cdots\\
&=\left(\prod_{\delta\in\mathcal O^+}\frac{1-q^{d\delta}}{1-q^{\delta}}\right)\left(\prod_{\delta\in S_1}\frac{1-q^{\delta}}{1-q^{d\delta}}\right)\left(\prod_{\delta\in S_2}\frac{1-q^{\delta}}{1-q^{d\delta}}\right)\left(\prod_{\delta\in S_3}\frac{1-q^{\delta}}{1-q^{d\delta}}\right)\cdots\\
&=\prod_{\delta\in\mathcal O^+}\left(\sum_{i=0}^{d-1}q^{i\delta}\right)\prod_{\delta\in\cup_{n=1}^{\infty}S_n}\left(\sum_{i=0}^{d-1}q^{i\delta}\right)^{-1}\\
&=\prod_{\delta\in S}\left(\sum_{i=0}^{d-1}q^{i\delta}\right).
\end{align*}
Note that 
\begin{align*}
\prod_{\delta\in\mathfrak a\setminus(d)}\frac{1-q^{d^n\delta}}{1-q^{d^{n+1}\delta}}=\prod_{\delta\in\mathfrak a(d^n)\setminus(d^{n+1})}\frac{1-q^{\delta}}{1-q^{d\delta}}=\prod_{\delta\in S_n}\frac{1-q^{\delta}}{1-q^{d\delta}}.
\end{align*}
Since $\sum_{\delta\in\mathcal O^+}p(''\mathcal O^+\setminus\mathfrak a'',\delta)q^{\delta}=\prod_{\delta\notin\mathfrak a}\frac1{1-q^{\delta}}$ and $\sum_{\delta\in\mathcal O^+}p(''S''(\le d-1),\delta)q^{\delta}=\prod_{\delta\in S}(\sum_{i=0}^{d-1}q^{i\delta})$, we have $p(''\mathcal O^+\setminus\mathfrak a'',\delta)=p(''S''(\le d-1),\delta)$ for all $\delta\in\mathcal O^+$.
\end{proof}

\begin{ex}
Let $K=\mathbb Q(\sqrt2)$, $\mathfrak a=(\sqrt2)$ and $d=2$. Note that $\mathfrak a$ is a prime ideal. Then we have
\begin{align*}
S&=\mathcal O^+\setminus\cup_{j=0}^{\infty}\left((\sqrt2)(2^j)\setminus (2^{j+1})\right)\\
&=\left(\mathcal O^+\setminus(\sqrt2)\right)\cup\left((2)\setminus(2\cdot\sqrt2)\right)\cup\left((2^2)\setminus(2^2\cdot\sqrt2)\right)\cup\cdots\\
&=\mathcal O^+\setminus\left\{\delta\in\mathcal O^+|\nu_{(\sqrt2)}(\delta)~{\rm is~odd}\right\}\\
&=\left\{\delta\in\mathcal O^+|\nu_{(\sqrt2)}(\delta)~{\rm is~even}\right\}.
\end{align*} 
The following list presents $p(6+2\sqrt2)$, $p(''{\mathcal O^+}\setminus(\sqrt2)'',6+2\sqrt2)$, $p(''S''(\le1),6+2\sqrt2)$ and tabulates the actual partitions. 
\begin{align*}
&p(6+2\sqrt2)=12:&&(6+2\sqrt2),~(5+2\sqrt2,1),~(4+2\sqrt2,2),~(4+2\sqrt2,1,1),\\
&&&(3+2\sqrt2,3),~(3+2\sqrt2,2,1),~(3+2\sqrt2,1,1,1),\\
&&&(4+\sqrt2,2+\sqrt2),~(3+\sqrt2,3+\sqrt2),~(3+\sqrt2,2+\sqrt2,1),\\
&&&(2+\sqrt2,2+\sqrt2,2),~(2+\sqrt2,2+\sqrt2,1,1);\\
&p(''{\mathcal O^+}\setminus(\sqrt2)'',6+2\sqrt2)=4:&&(5+2\sqrt2,1),~(3+2\sqrt2,3),~(3+2\sqrt2,1,1,1),\\
&&&(3+\sqrt2,3+\sqrt2);\\
&p(''S''(\le1),6+2\sqrt2)=4:&&(6+2\sqrt2),~(5+2\sqrt2,1),~(3+2\sqrt2,3),~(3+2\sqrt2,2,1).
\end{align*}
\end{ex}

\begin{ex}
Let $K=\mathbb Q(\sqrt2)$, $\mathfrak a=(3\sqrt2)$ and $d=6$. Note that $(3\sqrt2)=(3)(\sqrt2)$ is the factorization. Then we have
\begin{align*}
S&=\mathcal O^+\setminus\cup_{j=0}^{\infty}\left((3\sqrt2)(6^j)\setminus (6^{j+1})\right)\\
&=\left(\mathcal O^+\setminus(3\sqrt2)\right)\cup\left((6)\setminus(6\cdot3\sqrt2)\right)\cup\left((6^2)\setminus(6^2\cdot3\sqrt2)\right)\cup\cdots\\
&=\mathcal O^+\setminus\left\{\delta\in\mathcal O^+|\nu_{(\sqrt2)}(\delta)~{\rm is~odd},~ \nu_{(3)}(\delta)\ge\frac{\nu_{(\sqrt2)}(\delta)+1}2\right\}\\
&=\left\{\delta\in\mathcal O^+|\nu_{(\sqrt2)}(\delta)~{\rm is~even~or~} \nu_{(3)}(\delta)<\frac{\nu_{(\sqrt2)}(\delta)+1}2\right\}.
\end{align*}
\end{ex}

Let ${\mathcal{O}^+}_{d}$ denote the set of all totally positive integers not divisible by $d$. If $\mathfrak a=(d)$, then $\mathcal O^+\setminus\mathfrak a={\mathcal{O}^+}_{d}$ and $S={\mathcal{O}^+}$. Thus we have the following corollary.

\begin{cor}\label{glg}
For all $\delta\in\mathcal{O}^+$ and $d\in\mathbb Z^+$, we have $p({''{{\mathcal{O}^+}_{d}}''},\delta)=p(''{\mathcal{O}^+}''(\le d-1),\delta)$.
\end{cor}

Corollary \ref{glg} is a generalization of the Glaisher Theorem to totally real number fields. If $d=2$, then this corollary generalizes the Euler Theorem. 

\begin{ex}
Consider the partitions of $7+4\sqrt2$ over $\mathbb Q(\sqrt2)$. The following list presents $p(7+4\sqrt2)$, $p(''{\mathcal O^+}_2'',7+4\sqrt2)$, $p(''{\mathcal{O}^+}''(\le1),7+4\sqrt2)$ and tabulates the actual partitions. 
\begin{align*}
&p(7+4\sqrt2)=6:&&(7+4\sqrt2),~(6+4\sqrt2,1),~(5+3\sqrt2,2+\sqrt2),\\
&&&(4+2\sqrt2,3+2\sqrt2),~(3+2\sqrt2,3+2\sqrt2,1),\\
&&&(3+2\sqrt2,2+\sqrt2,2+\sqrt2),\\
&p(''{\mathcal O^+}_2'',7+4\sqrt2)=4:&&(7+4\sqrt2),~(5+3\sqrt2,2+\sqrt2),~(3+2\sqrt2,3+2\sqrt2,1),\\
&&&(3+2\sqrt2,2+\sqrt2,2+\sqrt2),\\
&p(''{\mathcal{O}^+}''(\le1),7+4\sqrt2)=4:&&(7+4\sqrt2),~(6+4\sqrt2,1),~(5+3\sqrt2,2+\sqrt2),\\
&&&(4+2\sqrt2,3+2\sqrt2).
\end{align*}
\end{ex}

\section{chain partitions}
The difference between two distinct parts of a partition of a natural number is also a natural number or $0$. However, in totally real number fields, the difference between some two parts of a partition of a totally positive algebraic integer may be neither totally positive nor $0$. We use the notation $\alpha\succeq0$ if and only if $\alpha$ is a totally positive integer or $0$. A partition $\lambda=(\lambda_1,\lambda_2,\ldots,\lambda_r)$ of a totally positive algebraic integer $\delta$ over $K$ is a $\textit chain~partition$ if and only if $\lambda_{i+1}-\lambda_i\succeq0$ for all $i=1,2,\ldots, r-1$. Let $p(\succeq,\delta)$ be the number of chain partitions of $\delta$. Here, we define $p(\succeq,m,\delta)$ as the number of chain partitions of $\delta$ with at most $m$ parts. 

\begin{thm}
The number of chain partitions $p(\succeq,m,\delta)$ of $\delta$ with at most $m$ parts is equal to that of solutions to $\delta=x_1+2x_2+\cdots+mx_m$ with $x_i\succeq0$ for all $m\in\mathbb Z^+$. As a consequence we have 
\begin{align*}
\sum_{\delta\in\mathcal O^+\cup\{0\}}p(\succeq,\delta)q^{\delta}=\prod_{m=1}^{\infty}\left(\sum_{\alpha\in\mathcal O^+\cup\{0\}}q^{m\alpha}\right).
\end{align*}
\end{thm}
\begin{proof}
Let 
\begin{align*}
\sum_{\gamma\in\mathcal O^+\cup\{0\}}c_{\gamma}q^{\gamma}
=&\left(\sum_{\alpha_1\in\mathcal O^+\cup\{0\}}1q^{\alpha_1}\right)\left(\sum_{\alpha_2\in\mathcal O^+\cup\{0\}}1q^{2\alpha_2}\right)\cdots\left(\sum_{\alpha_{m}\in\mathcal O^+\cup\{0\}}1q^{m\alpha_{m}}\right).
\end{align*}
Then $c_{\gamma}=\sum_{\alpha_1+2\alpha_2+\cdots+m\alpha_m=\gamma}1$, and thus, $c_{\gamma}$ is the number of solutions to $\delta=x_1+2x_2+\cdots+mx_m$ with $x_i\succeq0$. Let $\lambda=(\lambda_1,\lambda_2,\ldots,\lambda_r)$ be a chain partition of $\delta$ with $r\le m$. Let $\lambda_1'=\lambda_1$ and $\lambda_i'=\lambda_i-\lambda_{i-1}$ for all $2\le i\le r$. Since $\lambda_i'\in\mathcal O^+\cup\{0\}$ and
\begin{align*}
\delta=&\lambda_1+\lambda_2+\cdots+\lambda_r\\
=&r\lambda_1+(r-1)(\lambda_2-\lambda_1)+(r-2)(\lambda_3-\lambda_2)+\cdots+1(\lambda_r-\lambda_{r-1})\\
=&0+0+\cdots+0+r\lambda_1'+(r-1)\lambda_2'+\cdots+1\lambda_r',
\end{align*}
we have $\lambda_1'\ne0$ and $(\lambda_r',\lambda_{r-1}',\ldots,\lambda_1',0,0,\ldots,0)$ is a solution to $\delta=x_1+2x_2+\cdots+mx_m$ with $x_i\succeq0$. Note that the number of $0$ after $\lambda_1'$ is $m-r$. Let $(\mu_1,\mu_2,\ldots,\mu_m)$ be a solution to $\delta=x_1+2x_2+\cdots+mx_m$ with $x_i\succeq0$. Let $\mu_i'=\sum_{k=i}^r\mu_k$ for all $1\le i\le r$ where $r$ is the largest number such that $\mu_r\ne0$ for $1\le r\le m$. Since $\mu_i'-\mu_{i+1}'=\mu_i\succeq0$ for all $1\le i\le r-1$ and $\delta=\mu_1+2\mu_2+\cdots+r\mu_r=\mu_1'+\mu_2'+\cdots+\mu_r'$, we can see $(\mu_r',\mu_{r-1}',\ldots,\mu_1')$ is a chain partition of $\delta$. Let 
\begin{align*}
\phi(\lambda_1,\lambda_2,\ldots,\lambda_r)=(\lambda_r-\lambda_{r-1},\lambda_{r-1}-\lambda_{r-2},\ldots,\lambda_2-\lambda_1,\lambda_1,0,0,\ldots,0)
\end{align*}
and 
\begin{align*}
\psi(\mu_1,\mu_2,\ldots,\mu_m)=\left(\mu_r,\sum_{k=r-1}^r\mu_k,\sum_{k=r-2}^r\mu_k,\ldots,\sum_{k=1}^r\mu_k\right)
\end{align*}
where $r$ is the largest number such that $\mu_r\ne0$ for $1\le r\le m$. Since 
\begin{align*}
\psi(\phi(\lambda_1,\lambda_2,\ldots,\lambda_r))
=&\psi(\lambda_r-\lambda_{r-1},\lambda_{r-1}-\lambda_{r-2},\ldots,\lambda_2-\lambda_1,\lambda_1,0,0,\ldots,0)\\
=&\left(\lambda_1,(\lambda_2-\lambda_1)+\lambda_1,\sum_{k=2}^3(\lambda_k-\lambda_{k-1})+\lambda_1,\ldots,\sum_{k=2}^r(\lambda_k-\lambda_{k-1})+\lambda_1\right)\\
=&(\lambda_1,\lambda_2,\ldots,\lambda_r)
\end{align*}
and 
\begin{align*}
&\phi(\psi(\mu_1,\mu_2,\ldots,\mu_m))\\
=&\phi\left(\mu_r,\sum_{k=r-1}^r\mu_k,\sum_{k=r-2}^r\mu_k,\ldots,\sum_{k=1}^r\mu_k\right)\\
=&\left(\sum_{k=1}^r\mu_k-\sum_{k=2}^r\mu_k,\sum_{k=2}^r\mu_k-\sum_{k=3}^r\mu_k,\ldots,\sum_{k=r-1}^r\mu_k-\sum_{k=r}^r\mu_k,\mu_r,0,0,\ldots,0\right)\\
=&(\mu_1,\mu_2,\ldots,\mu_r,0,0,\ldots,0)\\
=&(\mu_1,\mu_2,\ldots,\mu_m),
\end{align*}
$p(\succeq,m,\delta)$ equals the number of solutions to $\delta=x_1+2x_2+\cdots+mx_m$ with $x_i\succeq0$ for all $m\in\mathbb Z^+$. Therefore, 
\begin{align*}
\sum_{\delta\in\mathcal O^+\cup\{0\}}p(\succeq,m,\delta)q^{\delta}=\prod_{i=1}^m\left(\sum_{\alpha\in\mathcal O^+\cup\{0\}}q^{i\alpha}\right).
\end{align*}
Clearly we have $\lim_{m\rightarrow\infty}p(\succeq,m,\delta)=p(\succeq,\delta)$. Thus, we have
\begin{align*}
\sum_{\delta\in\mathcal O^+\cup\{0\}}p(\succeq,\delta)q^{\delta}=\prod_{m=1}^{\infty}\left(\sum_{\alpha\in\mathcal O^+\cup\{0\}}q^{m\alpha}\right).
\end{align*}
\end{proof}

\begin{ex}
Consider the partitions of $7+2\sqrt3$ over $\mathbb Q(\sqrt3)$. The following list presents $p(\succeq,m,7+2\sqrt3)$, the solutions to $7+2\sqrt3=x_1+2x_2+\cdots+mx_m$ and tabulates the actual partitions. 
\begin{align*}
&p(\succeq,1,7+2\sqrt3)=1:&&(7+2\sqrt3),\\
&p(\succeq,2,7+2\sqrt3)=3:&&(6+2\sqrt3,1),~(5+\sqrt3,2+\sqrt3),~(4+\sqrt3,3+\sqrt3),\\
&p(\succeq,3,7+2\sqrt3)=2:&&(5+2\sqrt3,1,1),~(3+\sqrt3,3+\sqrt3,1),\\
&7+2\sqrt3=x:&&x=7+2\sqrt3,\\
&7+2\sqrt3=x+2y:&&(x,y)=(5+2\sqrt3,1),(1,3+\sqrt3),(3,2+\sqrt3),\\
&7+2\sqrt3=x+2y+3z:&&(x,y,z)=(4+2\sqrt3,0,1),(0,2+\sqrt3,1).
\end{align*}
\end{ex}

\begin{remark}
Let $K=\mathbb Q(\sqrt3)$. The ideal $(1+\sqrt3)$ of $\mathcal O_K$ is the only prime ideal of $\mathcal O$ containing $2$. An element $a+b\sqrt3$ of $\mathcal O$ does not belong to $(1+\sqrt3)$ if and only if $a+b$ is odd. We compute the number of the two kinds of chain partition of $\delta$ for $\delta=7+2\sqrt3, 9+2\sqrt3$. The first one is the chain partition with distinct parts and the next one is the chain parition whose $\delta=7+2\sqrt3$, these two numbers are both $4$, but when $\delta=9+2\sqrt3$, these numbers are $7$ and $5$ respectively. Thus the Euler Theorem does not hold in this cartegory.
\end{remark}


\begin{thebibliography}{12345}
\bibitem{and11} Andrews, G. E., {\em The Theory of Partitions}, Cambridge University Press, 1984.
\bibitem{euler} Euler, L., {\em Introductio in Analysin Infinitorum}, Chapter 16. Marcum-Michaelem Bousquet, Lausannae, 1748.
\bibitem{glais} Glaisher, J. W. L., {\em A Theorem in Partitions}, Messenger of Math. 12, 158-170, 1883.
\bibitem{kim} Kim, B. M., {\em Positive Universal Forms over Totally Real Fields}, Ph.D. Thesis, Seoul National University, 1997.
\bibitem{mit} Mitsui, T., {\em On the Partition Problem in an Algebraic Number Field}, Tokyo J. Math. Vol. 1, No. 2, 189-236, 1978.
\bibitem{rad} Rademacher, H., {\em Additive algebraic number theory}, Proc. Internat. Congress of Math.,
Cambridge, Mass., Vol. 1, 356-362, 1950. 
\end{thebibliography}
\end{document}